\documentclass[11pt,reqno]{amsart}

\usepackage{xspace}
\usepackage{tikz}
\usepackage{amsmath,amsthm,amsfonts,amssymb,latexsym,mathrsfs,color}
\usepackage{hyperref}

\newtheorem{theorem}{Theorem}
\newtheorem{cor}[theorem]{Corollary}
\newtheorem{prop}[theorem]{Proposition}

\newtheorem{lemma}[theorem]{Lemma}
\newtheorem{definition}[theorem]{Definition}
\newtheorem{ex}[theorem]{Example}
\newtheorem{problem}[theorem]{Problem}

\newcommand{\evenp}{{\rm evenp\,}}
\newcommand{\zero}{{\rm zerop\,}}
\newcommand{\st}{{\rm singleton\,}}

\newcommand{\child}{{\rm child\,}}
\newcommand{\SSS}{\mathcal{S}}
\newcommand{\M}{{\rm M\,}}
\newcommand{\leaff}{{\rm Leaf\,}}
\newcommand{\Child}{{\rm Child\,}}
\newcommand{\asc}{{\rm asc\,}}

\newcommand{\CPK}{{\rm CPK\,}}
\newcommand{\cpke}{{\rm cpk^{e}\,}}
\newcommand{\cpko}{{\rm cpk^{o}\,}}
\newcommand{\sech}{\ensuremath{\mathrm{sech\ }}}

\newcommand{\sn}{{\rm sn\,}}
\newcommand{\cn}{{\rm cn\,}}
\newcommand{\dn}{{\rm dn\,}}

\newcommand{\des}{{\rm des\,}}

\newcommand{\NAP}{\mathcal{NAP}}

\newcommand{\msn}{\mathfrak{S}_n}

\newcommand{\lrf}[1]{\lfloor #1\rfloor}

\title{On the unimodality of the Taylor expansion coefficients of Jacobian elliptic functions}
\author[S.-M.~Ma]{Shi-Mei Ma}
\address{School of Mathematics and Statistics,
        Northeastern University at Qinhuangdao,
         Hebei 066004, P.R. China}
\email{shimeimapapers@163.com (S.-M. Ma)}
\author[J. Ma]{Jun Ma}
\address{Department of mathematics, Shanghai jiao tong university, Shanghai, P.R. China}
\email{majun904@sjtu.edu.cn (J. Ma)}
\author[Y.-N. Yeh]{Yeong-Nan Yeh}
\address{Institute of Mathematics,
        Academia Sinica, Taipei, Taiwan}
\email{mayeh@math.sinica.edu.tw (Y.-N. Yeh)}
\author[R.R. Zhou]{Roberta R. Zhou}
\address{School of Mathematics and Statistics,
        Northeastern University at Qinhuangdao,
         Hebei 066004, P.R. China}
\email{zhourui@neuq.edu.cn (R.R. Zhou)}

\begin{document}
\begin{abstract}
The Jacobian elliptic functions are standard forms of elliptic functions, and they were independently introduced by C.G.J. Jacobi and N.H. Abel.
In this paper, we study the unimodality of Taylor expansion coefficients
of the Jacobian elliptic functions $\sn(u,k)$ and $\cn(u,k)$. By using the theory of $\gamma$-positivity, we obtain that the Taylor expansion coefficients of
$\sn(u,k)$ are symmetric and unimodal, and that of $\cn(u,k)$ are unimodal and alternatingly increasing.
\end{abstract}


\maketitle

\section{Introduction}
Elliptic integrals were first investigated in works of scholars at the end of the 17th century to the beginning of the 19th century: L. Euler, A. Legendre and C.G.J. Jacobi,
The {\it elliptic integral of the first kind} is given as follows:
$$u=\int_{0}^{x}\frac{d t}{\sqrt{(1-t^2)(1-k^2t^2)}},$$
where $k\in (0,1)$ is the {\it modulus}.
The {\it Jacobian elliptic function} $\sn(u,k)$ is the inverse to this elliptic integral, i.e., $x=\sn(u,k)$.
This inversion problem was solved independently by C.G.J. Jacobi~\cite{Jacobi} and, in a slightly different form, by N.H. Abel.
The other two Jacobian elliptic functions are defined by
$\cn(u,k)=\sqrt{1-\sn^2(u,k)}$ and $\dn(u,k)=\sqrt{1-k^2\sn^2(u,k)}$.

The three Jacobian elliptic functions are
connected by the differential system
\begin{equation}\label{Abel-diff}
\left\{
  \begin{array}{ll}
    \frac{d}{du}\sn(u,k)=\cn(u,k)\dn(u,k), \\
\frac{d}{du}\cn(u,k)=-\sn(u,k)\dn(u,k),\\
\frac{d}{du}\dn(u,k)=-k^2\sn(u,k)\cn(u,k).
  \end{array}
\right.
\end{equation}
When $k=0$ or $k=1$, the Jacobian elliptic functions degenerate into trigonometric or hyperbolic functions:
\begin{align*}
\sn(u,0)&=\sin u,~\cn(u,0)=\cos u,~\dn(u,0)=1,\\
\sn(u,1)&=\tanh u,~\cn(u,1)=\dn(u,1)=\sech u.
\end{align*}
These functions appear in a variety of problems in physics and engineering,
and they have been extensively studied in mathematical physics, algebraic geometry, combinatorics and number
theory (see~\cite{Carlson,Conrad02,Ismail98,Viennot80} for instance).

Following Viennot~\cite{Viennot80},
we define $J_n(x)$ as the Taylor expansion coefficients of the Jacobian elliptic functions, i.e.,
\begin{align*}
\sn(u,k)&=\sum_{n\geq 0}(-1)^nJ_{2n+1}(k^2)\frac{u^{2n+1}}{(2n+1)!},\\
\cn(u,k)&=1+\sum_{n\geq 1}(-1)^nJ_{2n}(k^2)\frac{u^{2n}}{(2n)!},\\
\dn(u,k)&=1+\sum_{n\geq 1}(-1)^nk^{2n}J_{2n}(1/k^2)\frac{u^{2n}}{(2n)!}.
\end{align*}
It is well known that $\deg J_n(x)=\lrf{(n-1)/2}$ (see~\cite{Viennot80} for instance).
The first few $J_n(x)$ are listed as follows:
\begin{align*}
J_1(x)&=J_2(x)=1,~J_3(x)=1+x,~J_4(x)=1+4x,\\
J_5(x)&=1+14x+x^2,~J_6(x)=1+44x+16x^2,\\
J_7(x)&=1+135x+135x^2+x^3,~J_8(x)=1+408x+912x^2+64x^3.
\end{align*}

Let $\operatorname{D}_J$ be the derivative operator, acting on commuting variables $\{x,y,z\}$, that is given by
\begin{equation}\label{diff-elliptic}
\operatorname{D}_J=yz\frac{\partial}{\partial x}+xz\frac{\partial}{\partial y}+xy\frac{\partial}{\partial z}.
\end{equation}
Following an approach due to Schett for computing the
Taylor expansion coefficients of Jacobian elliptic functions~\cite{Schett76},
Dumont~\cite{Dumont79} obtained a connection between $\operatorname{D}_J^n(x)$ and peak statistics of permutations.
In~\cite{Viennot80}, Viennot presented a combinatorial interpretation of $J_n(x)$, and by using~\eqref{Abel-diff}, he also provided
several convolution formulas for $J_n(x)$. In~\cite{Flajolet89}, Flajolet and Fran\c{c}on gave another combinatorial interpretation of
$J_n(x)$ by using continued fractions.

Let us now recall a classical result of Viennot~\cite{Viennot80}.
\begin{prop}[{\cite[Remarque~15]{Viennot80}}]
The polynomial $J_{2n+1}(x)$ is symmetric for any $n\geq 0$.
\end{prop}

This paper is motivated by the following problem.
\begin{problem}
For any $n\geq 0$, whether the polynomial $J_n(x)$ is unimodal?
\end{problem}

Let $f(x)=\sum_{i=0}^nf_ix^i\in \mathbb{R}[x]$.
The polynomial $f(x)$ {\it unimodal} if there exists an index $m$ such that $f_0\leq f_1\leq\cdots \leq f_{m}\geq f_{m+1}\geq\cdots\geq f_n$.
We say that $f(x)$ is {\it alternatingly increasing} if
$$f_0\leq f_n\leq f_1\leq f_{n-1}\leq\cdots f_{\lrf{\frac{n+1}{2}}}.$$
Clearly, alternatingly increasing property is a stronger property than unimodality.
We say that $f(x)$ is {\it symmetric}
if $f_i=f_{n-i}$ for $0\leq i\leq n$, and the number $\lrf{n/2}$ is called the center of symmetry.
If $f(x)$ is a symmetric polynomial, then it can be expanded as
$$f(x)=\sum_{k=0}^{\lrf{{n}/{2}}}\gamma_kx^k(1+x)^{n-2k}.$$ The vector $(\gamma_0,\gamma_1,\ldots,\gamma_{\lrf{n/2}})$ is known as the {\it $\gamma$-vector} of $f(x)$.
We say that $f(x)$ is {\it $\gamma$-positive} if $\gamma_k\geq 0$ for all $0\leq k\leq \lrf{n/2}$ (see~\cite{Branden08,Lin15,Ma19}).
It is clear that $\gamma$-positivity of $f(x)$ implies that $f(x)$ is unimodal and symmetric.
The $\gamma$-coefficients of $\gamma$-positive polynomials often have nice combinatorial interpretations.
See~\cite{Athanasiadis17,Lin15} for a recent
comprehensive survey on this subject.

We now present the first main result of this paper.
\begin{theorem}\label{thm01}
The polynomial $J_{2n+1}(x)$ is $\gamma$-positive for any $n\geq 0$. Thus $J_{2n+1}(x)$ is symmetric and unimodal.
\end{theorem}

We now recall an elementary result.
\begin{prop}[{\cite{Beck2015,Branden18}}]\label{prop01}
Let $f(x)$ be a polynomial of degree $n$.
There is a unique symmetric decomposition $(a(x),b(x))$ such that $f(x)= a(x)+xb(x)$, where $a(x)$ and $b(x)$ are symmetric polynomials
satisfying $a(x)=x^n a(\frac{1}{x})$ and $b(x)=x^{n-1}b(\frac{1}{x})$. More precisely,
\begin{equation*}\label{ax-bx-prop01}
a(x)=\frac{f(x)-x^{n+1}f(1/x)}{1-x},~b(x)=\frac{x^nf(1/x)-f(x)}{1-x}.
\end{equation*}
\end{prop}

When $f(x)$ is symmetric, we have $a(x)=f(x)$ and $b(x)=0$.
As pointed out by Br\"and\'en and Solus~\cite{Branden18},
if $(a(x),b(x))$ is the symmetric decomposition of $f(x)$,
then $f(x)$ is alternatingly increasing if and only if $a(x)$ and $b(x)$ are both unimodal
and have nonnegative coefficients.
We now introduce the following definition.
\begin{definition}\label{def-bi}
Let $(a(x),b(x))$ be the symmetric decomposition of $f(x)$. If $a(x)$ and $b(x)$ are both $\gamma$-positive, then we say that
$f(x)$ is bi-$\gamma$-positive.
\end{definition}

Note that bi-$\gamma$-positivity of $f(x)$ implies that $f(x)$ is alternatingly increasing.
The second main result of this paper is the following.
\begin{theorem}\label{thm02}
The polynomial $J_{2n}(x)$ is bi-$\gamma$-positive for any $n\geq 0$. Thus $J_{2n}(x)$ is alternatingly increasing and unimodal.
\end{theorem}

In the next section, we give an overview of some
results that have surfaced in the process of trying to understand $J_n(x)$.
\section{Notation and Preliminaries}
Let $\operatorname{D}_J$ be given in~\eqref{diff-elliptic}.
It follows easily by induction that there exist polynomials $S_n(p,q,r)$ of degree $\lrf{n/2}$, homogeneous in the variables $p,q,r$ such that
\begin{equation}\label{recurrence001}
\operatorname{D}_J^{2n}(x)=xS_{2n}(x^2,y^2,z^2),~\operatorname{D}_J^{2n+1}(x)=yzS_{2n+1}(x^2,y^2,z^2).
\end{equation}
More precisely, there exist nonnegative integers $s_{n,i,j}$ such that
\begin{equation}\label{recurrence01}
\begin{gathered}
\operatorname{D}_J^{2n}(x)=\sum_{i,j\geq 0}s_{2n,i,j}x^{2i+1}y^{2j}z^{2n-2i-2j},\\
\operatorname{D}_J^{2n+1}(x)=\sum_{i,j\geq 0}s_{2n+1,i,j}x^{2i}y^{2j+1}z^{2n-2i-2j+1}.
\end{gathered}
\end{equation}
Thus
\begin{equation}\label{recurrence0001}
S_n(p,q,r)=\sum_{i,j\geq 0}s_{n,i,j}p^{i}q^{j}r^{n-i-j}.
\end{equation}
In order to compute the Taylor expansion coefficients of Jacobian elliptic functions,
Schett~\cite{Schett76} initiated the study of the numbers $s_{n,i,j}$ in a slightly different form. Subsequently, Dumont~\cite[Eq.~(2)]{Dumont79} derived the following recurrence system:
\begin{equation*}\label{snijrecurrence}
\begin{gathered}
s_{2n,i,j}=(2j+1)s_{2n-1,i,j}+(2i+2)s_{2n-1,i+1,j-1}+(2n-2i-2j+1)s_{2n-1,i,j-1},\\
s_{2n+1,i,j}=(2i+1)s_{2n,i,j}+(2j+2)s_{2n,i-1,j+1}+(2n-2i-2j+2)s_{2n,i-1,j}.
\end{gathered}
\end{equation*}

Let $\msn$ be the symmetric group of all permutations of the set $[n]=\{1,2,\ldots,n\}$ and
let $\pi=\pi(1)\pi(2)\cdots\pi(n)\in\msn$.
A value $i\in[n]$ is called a {\it cycle peak} of $\pi$ if $\pi^{-1}(i)<i>\pi(i)$.
Let $\CPK^o(\pi)$ and $\CPK^e(\pi)$ denote the set of odd and even cycle peaks of $\pi$, respectively.
For example, if $\pi=(1,3,5)(2,8,4,7,6,9)$, then $\CPK^o(\pi)=\{5,7,9\}$ and $\CPK^e(\pi)=\{8\}$.
Let $\cpko(\pi)=\#\CPK^o(\pi)$ and let $\cpke(\pi)=\#\CPK^e(\pi)$.
The number of peaks in a permutation is an important
statistic in algebraic combinatorics. See, e.g.,~\cite{Dilks09,Stembridge97} and the references therein.

We now recall a remarkable result of Dumont.
\begin{prop}[{\cite{Dumont79}}]\label{prop01}
We have
\begin{equation}\label{Dumont-com}
s_{n,i,j}=\#\{\pi\in\msn\mid \cpko(\pi)=i, \cpke(\pi)=j\}.
\end{equation}
\end{prop}

We now define $$P_n(p,q)=\sum_{\pi\in\msn}p^{\cpko(\pi)}q^{\cpke(\pi)}.$$
Combining~\eqref{recurrence0001} and~\eqref{Dumont-com}, we see that $$P_n(p,q)=S_n(p,q,1).$$
The first few $P_n(p,q)$ are listed as follows:
\begin{align*}
P_1(p,q)&=1,~
P_2(p,q)=1+q,~
P_3(p,q)=1+q+4p,~
P_4(p,q)=1+14q+q^2+4p(1+q),\\
P_5(p,q)&=1+14q+q^2+44p(1+q)+16p^2,\\
P_6(p,q)&=1+135q+135q^2+q^3+p(44+328q+44q^2)+16p^2(1+q).
\end{align*}

The following two lemmas is fundamental.
\begin{lemma}[{\cite[Corollary~1]{Dumont79}}]\label{Dumont}
For $n\geq 0$, we have
$$J_{2n}(x)=P_{2n}(x,0)=P_{2n-1}(x,0),$$
$$J_{2n+1}(x)=P_{2n+1}(0,x)=P_{2n}(0,x).$$
\end{lemma}

\begin{lemma}[{\cite[Eq.~(20)]{Viennot80}}]\label{Viennot}
Set $J_0(x)=1$. Then
for $n\geq 1$, we have
$$J_{2n}(x)=\sum_{i=0}^{n-1} \binom{2n-1}{2i}J_{2n-1-2i}(x)x^iJ_{2i}\left({1}/{x}\right),$$
$$J_{2n+1}(x)=\sum_{i=0}^{n} \binom{2n}{2i}J_{2n-2i}(x)x^iJ_{2i}\left({1}/{x}\right).$$
\end{lemma}

\section{Proof of the main results}\label{Section03:mainthm}
Let $\msn^{(i)}=\{\pi\in\msn\mid~\cpko(\pi)=i\}$. To prove Theorem~\ref{thm01},
we need the following lemma.
\begin{lemma}\label{thm2}
For $n\geq 1$, there are nonnegative integers $\gamma_{n,i,j}$ such that
\begin{equation}\label{gammapeak}
P_n(p,q)=\sum_{i=0}^{\lrf{(n-1)/2}}\sum_{j=0}^{\lrf{(n-2i)/4}}\gamma_{n,i,j}p^iq^j(1+q)^{\lrf{n/2}-i-2j}.
\end{equation}
Therefore, the polynomial $\sum_{\pi\in\msn^{(i)}}q^{\cpke(\pi)}$ is $\gamma$-positive for any $i\geq 0$.
\end{lemma}
\begin{proof}
For the derivative operator $\operatorname{D}_J$ given by~\eqref{diff-elliptic}, we now consider a change of variables.
Set
\begin{equation}\label{abcxyz}
a=y^2,b=z^2,c=yz.
\end{equation}
Then we have
$\operatorname{D}_J(x)=c,\operatorname{D}_J(a)=2xc,\operatorname{D}_J(b)=2xc,\operatorname{D}_J(c)=x(a+b)$.
Let $\operatorname{D}_G$ be the derivative operator, acting on commuting variables $\{x,a,b,c\}$, that is given by
\begin{equation}\label{diff-elliptic02}
\operatorname{D}_G=c\frac{\partial}{\partial x}+2xc\frac{\partial}{\partial a}+2xc\frac{\partial}{\partial b}+x(a+b)\frac{\partial}{\partial c}.
\end{equation}
Note that $D_{G}(x)=c,~D_{G}^2(x)=D_{G}(c)=x(a+b),~D_{G}^3(x)=c(a+b+4x^2)$.
Then by induction, it is a routine check to verify that there exist nonnegative integers $\gamma_{n,i,j}$ such that
\begin{equation}\label{recurrence002}
\begin{gathered}
D_{G}^{2n}(x)=x\sum_{i,j\geq 0}\gamma_{2n,i,j}x^{2i}c^{2j}(a+b)^{n-i-2j},\\
D_{G}^{2n+1}(x)=c\sum_{i,j\geq 0}\gamma_{2n+1,i,j}x^{2i}c^{2j}(a+b)^{n-i-2j}.
\end{gathered}
\end{equation}
It follows that
\begin{align*}
D_{G}^{2n+1}(x)&=D_{G}\left(\sum_{i,j\geq 0}\gamma_{2n,i,j}x^{2i+1}c^{2j}(a+b)^{n-i-2j}\right)\\
&=c\sum_{i,j\geq 0}\gamma_{2n,i,j}\left((2i+1)x^{2i}c^{2j}(a+b)^{n-i-2j}+2jx^{2i+2}c^{2j-2}(a+b)^{n-i-2j+1}\right)+\\
&c\sum_{i,j\geq 0}\gamma_{2n,i,j}4(n-i-2j)x^{2i+2}c^{2j}(a+b)^{n-i-2j-1},
\end{align*}
\begin{align*}
D_{G}^{2n+2}(x)&=D_{G}\left(\sum_{i,j\geq 0}\gamma_{2n+1,i,j}x^{2i}c^{2j+1}(a+b)^{n-i-2j}\right)\\
&=x\sum_{i,j\geq 0}\gamma_{2n+1,i,j}\left(2ix^{2i-2}c^{2j+2}(a+b)^{n-i-2j}+(2j+1)x^{2i}c^{2j}(a+b)^{n-i-2j+1}\right)+\\
&x\sum_{i,j\geq 0}\gamma_{2n+1,i,j}4(n-i-2j)x^{2i}c^{2j+2}(a+b)^{n-i-2j-1}.
\end{align*}
Therefore, the numbers $\gamma_{n,i,j}$ satisfy the system of recurrences
\begin{equation}\label{recurrence02}
\begin{gathered}
\gamma_{2n,i,j}=2(i+1)\gamma_{2n-1,i+1,j-1}+(2j+1)\gamma_{2n-1,i,j}+4(n-i-2j+1)\gamma_{2n-1,i,j-1},\\
\gamma_{2n+1,i,j}=(2i+1)\gamma_{2n,i,j}+2(j+1)\gamma_{2n,i-1,j+1}+4(n-i-2j+1)\gamma_{2n,i-1,j},
\end{gathered}
\end{equation}
with the initial conditions $\gamma_{1,0,0}=\gamma_{2,0,0}=1$.
Set $x^2=p,~y^2=a=q$ and $z=1$. Then $z^2=b=1$
and $c=yz=y, c^2=ab=q$. Substituting to~\eqref{recurrence01} and~\eqref{recurrence002}, we get the desired formula~\eqref{gammapeak} by comparison.
This completes the proof.
\end{proof}

We now define the numbers $t_{n,i,j}$ by
\begin{equation*}\label{tnij}
t_{n,i,j}=\frac{\gamma_{n,i,j}}{4^{i+j}}.
\end{equation*}
From~\eqref{recurrence02}, we see that the numbers $t_{n,i,j}$ satisfy the system of recurrences
\begin{equation}\label{recurrence03}
\begin{gathered}
t_{2n,i,j}=2(i+1)t_{2n-1,i+1,j-1}+(2j+1)t_{2n-1,i,j}+(n-i-2j+1)t_{2n-1,i,j-1},\\
t_{2n+1,i,j}=(2i+1)t_{2n,i,j}+2(j+1)t_{2n,i-1,j+1}+(n-i-2j+1)t_{2n,i-1,j},
\end{gathered}
\end{equation}
with initial conditions $t_{1,0,0}=t_{2,0,0}=1$ and $t_{1,i,j}=t_{2,i,j}=0$ for $(i,j)\neq (0,0)$.
Let $$t_{n}(x,y)=\sum_{i=0}^{\lrf{(n-1)/2}}\sum_{j=0}^{\lrf{(n-2i)/4}}t_{n,i,j}x^iy^j.$$
It is easy to verify that the polynomials $t_{n}(x,y)$ satisfy the system of recurrences
\begin{equation}\label{recurrence04}
\begin{gathered}
t_{2n}(x,y)=(1+(n-1)y)t_{2n-1}(x,y)+(2-x)y\frac{\partial}{\partial x}t_{2n-1}(x,y)+2y(1-y)\frac{\partial}{\partial y}t_{2n-1}(x,y),\\
t_{2n+1}(x,y)=(1+nx)t_{2n}(x,y)+(2-x)x\frac{\partial}{\partial x}t_{2n}(x,y)+2x(1-y)\frac{\partial}{\partial y}t_{2n}(x,y).
\end{gathered}
\end{equation}
The first few polynomials $t_n(x,y)$ are given as follows:
\begin{align*}
t_1(x,y)&=t_2(x,y)=1,~
t_3(x,y)=1+x,~
t_4(x,y)=1+x+3y,\\
t_5(x,y)&=1+11x+x^2+3y,~
t_6(x,y)=1+11x+x^2+33y+15xy,\\
t_7(x,y)&=1+102x+57x^2+x^3+33y+78xy.
\end{align*}

\noindent{\bf A proof
Theorem~\ref{thm01}:}
\begin{proof}
Combining Lemma~\ref{Dumont} and Lemma~\ref{thm2}, we get
\begin{equation}\label{J2n1gamma}
J_{2n+1}(x)=\sum_{\pi\in\msn^{(0)}}x^{\cpke(\pi)}=\sum_{j=0}^{\lrf{n/2}}\gamma_{2n+1,0,j}x^{j}(1+x)^{n-2j}.
\end{equation}
Thus $J_{2n+1}(x)$ is $\gamma$-positive. This completes the proof.
\end{proof}

Note that $$J_{2n}(x)=\sum_{i=0}^{n-1}s_{2n,i,0}x^i=\sum_{i=0}^{n-1}\gamma_{2n,i,0}x^i=\sum_{i=0}^{n-1}4^it_{2n,i,0}x^i.$$
In the rest of this section, we show the bi-$\gamma$-positivity of $J_{2n}(x)$.

Assume that
$$f(x)=\sum_{i=0}^{\lrf{{m}/{2}}}\gamma_ix^i(1+x)^{m-2i},~g(x)=\sum_{j=0}^{\lrf{{n}/{2}}}\mu_jx^j(1+x)^{n-2j}.$$
Note that
$$f(x)g(x)=\sum_{i,j}\gamma_i\mu_jx^{i+j}(1+x)^{m+n-2(i+j)}.$$
Thus if $f_1(x)$ and $f_2(x)$ are both $\gamma$-positive, then $f_1(x)f_2(x)$ is $\gamma$-positive.
It is clear that
if $f(x)$ is $\gamma$-positive and $g(x)$ is bi-$\gamma$-positive, then $f(x)g(x)$ is bi-$\gamma$-positive.

\noindent{\bf A proof
Theorem~\ref{thm02}:}
\begin{proof}
We do this by induction on $n$.
Note that
$J_0(x)=J_2(x)=1,J_4(x)=1+4x=(1+x)+3x$ and $J_6(x)=1+44x+16x^2=(1+29x+x^2)+x(15+15x)$.
Hence the result holds for $n\leq 3$. So we proceed to the inductive step.
Assume that $J_{2n}(x)$ are bi-$\gamma$-positive for all $0\leq n\leq m$, where $m\geq 3$.
It follow from Lemma~\ref{Viennot} that
\begin{equation}\label{Jm03}
J_{2m+2}(x)=\sum_{i=0}^{m} \binom{2m+1}{2i}J_{2m+1-2i}(x)x^iJ_{2i}\left({1}/{x}\right).
\end{equation}
Recall that $\deg J_m(x)=\lrf{(m-1)/2}$.
Then for any $1\leq i\leq m$,
suppose that
$$J_{2i}(x)=\sum_{k=0}^{\lrf{(i-1)/2}} a_{2i,k}x^k(1+x)^{i-1-2k}+x\sum_{k=0}^{\lrf{i/2}-1} b_{2i,k}x^k(1+x)^{i-2-2k}.$$
Hence
\begin{equation*}\label{J2n01}
x^{i}J_{2i}(1/x)=x\sum_{k=0}^{\lrf{(i-1)/2}} a_{2i,k}x^k(1+x)^{i-1-2k}+\sum_{k=0}^{\lrf{i/2}-1} b_{2i,k}x^{k+1}(1+x)^{i-2(k+1)}.
\end{equation*}
From~\eqref{J2n1gamma}, we see that
\begin{equation*}\label{J2n02}
J_{2m+1-2i}(x)=\sum_{j=0}^{\lrf{(m-i)/2}}\gamma_{2m+1-2i,0,j}x^j(1+x)^{m-i-2j}
\end{equation*}
for $0\leq i\leq m$.
Therefore, we have
\begin{align*}
\sum_{i=1}^{m} \binom{2m+1}{2i}J_{2m+1-2i}(x)x^{i}J_{2i}(1/x)&=\sum_{\ell=1}^{\lrf{m/2}}c_{m,\ell}x^{\ell}(1+x)^{m-2\ell}+\\
&x\sum_{\ell=0}^{\lrf{(m-1)/2}}d_{m,\ell}x^{\ell}(1+x)^{m-1-2\ell},
\end{align*}
where $$c_{m,\ell}=\sum_{i=1}^{m} \binom{2m+1}{2i}\sum_{\substack{j+k+1=\ell\\j,k\geq 0}}\gamma_{2m+1-2i,0,j}b_{2i,k},$$
$$d_{m,\ell}=\sum_{i=1}^{m} \binom{2m+1}{2i}\sum_{\substack{j+k=\ell\\ j,k\geq 0}}\gamma_{2m+1-2i,0,j}a_{2i,k}.$$

By using~\eqref{Jm03}, we get
\begin{align*}
J_{2m+2}(x)&=J_{2m+1}(x)+\sum_{i=1}^{m} \binom{2m+1}{2i}J_{2m+1-2i}(x)x^{i}J_{2i}(1/x)\\
&=A_{2m+2}(x)+xB_{2m+2}(x),
\end{align*}
where
\begin{align*}
A_{2m+2}(x)&=\sum_{j=0}^{\lrf{m/2}}\gamma_{2m+1,0,j}x^j(1+x)^{m-2j}+\sum_{\ell=1}^{\lrf{m/2}}c_{m,\ell}x^{\ell}(1+x)^{m-2\ell},\\
&B_{2m+2}(x)=\sum_{\ell=0}^{\lrf{(m-1)/2}}d_{m,\ell}x^{\ell}(1+x)^{m-1-2\ell}.
\end{align*}
Clearly, $A_{2m+2}(x)$ and $B_{2m+2}(x)$ are both $\gamma$-positive.
Therefore, the polynomial $J_{2m+2}(x)$ is bi-$\gamma$-positive. This completes the proof.
\end{proof}

As an illustration, consider the following example.
\begin{ex}
Consider the polynomial $J_8(x)$.
Recall that
\begin{equation*}
J_{8}(x)=\sum_{i=0}^{3} \binom{7}{2i}J_{7-2i}(x)x^iJ_{2i}\left({1}/{x}\right).
\end{equation*}
It is easy to verify that $a_{2,0}=1,b_{2,0}=0,a_{4,0}=1,b_{4,0}=3,a_{6,0}=1,a_{6,1}=27,b_{6,0}=15$.
Note that $J_7(x)=(1+x)^3+132x(1+x)$.
Then
\begin{align*}
J_8(x)&=\left\{(1+x)^3+132x(1+x)+\binom{7}{4}3x(1+x)+\binom{7}{6}15x(1+x)\right\}+\\
&x\left\{\binom{7}{2}((1+x)^2+12x)+\binom{7}{4}(1+x)^2+\binom{7}{6}((1+x)^2+27x)\right\}\\
&=\left\{(1+x)^3+342x(1+x)\right\}+x\left\{63(1+x)^2+441x \right\}.
\end{align*}
\end{ex}

In the same way as in the proof of Theorem~\ref{thm02}, it is routine to check the following result.
\begin{prop}
Let $\{g_n(x)\}_{n\geq 0}$ be a sequence of $\gamma$-positive polynomials, and $\deg g_n(x)=n$. Let $(M_{i,j})_{i\geq 0,j\geq0}$ be an array of nonnegative real numbers.
For $n\geq 0$, we define $$f_{n+1}(x)=\sum_{i=0}^nM_{n,i}g_{n-i}(x)x^if_i(1/x),$$
where $f_0(x)=1$.
Then $f_n(x)$ is a bi-$\gamma$-positive polynomial for any $n\geq 0$. And so, $f_n(x)$ is unimodal,
\end{prop}

\section{The combinatorial interpretations of $\gamma$-coefficients}\label{Section05}

In the past decades, the bijections between $\msn$ and increasing trees on $n+1$ vertices are repeatedly discovered (see~\cite[Section~1.5]{Stan11} for instance).
It is natural to explore a combinatorial interpretation of the numbers $s_{n,i,j}$ and $\gamma_{n,i,j}$ in terms of some statistics on increasing trees

Let $[n]_0=\{0,1,\ldots,n\}$.
We define an {\it increasing tree} as an unordered tree with vertices set $[n]_0$, rooted at $0$ and
the labels increase along each path from the root. Denote by $\mathcal{T}_n$ the set of
increasing trees with $n+1$ vertices.
Let $T\in\mathcal{T}_n$. For nodes $u$ and $v$ in $T$,
we say that $v$ is the {\it child} of $u$ or $u$ is the {\it predecessor} of $v$ if $u$ is the first node following
$v$ in the unique path from $v$ to the root $0$, and we write as $u=p_T(v)$.
For any vertex $u$, let $\Child_T(u)$ be the set of children of the vertex $u$, and let $\child_T(u)=\#\Child_T(u)$.
If $\child_T(u)=0$, we say that $u$ is a {\it leaf} of $T$.
Let $\leaff_T$ be the set of leaves of $T$.

A {\it partition} $\sigma$ of $[n]_0$ is a collection of nonempty
disjoint subsets $B_1,\ldots,B_t$, called {\it blocks}, whose union is $[n]_0$.
Then $\sigma$ is a {\it matching} if each block of $\sigma$ contains only one or two elements.
A {\it singleton} is a block with only one element. A {\it sub-matching} $\M$ of $[n]_0$ is
the union of all non-singletons in a matching of $[n]_0$, i.e., the cardinality of each block of $\M$ is exactly 2.
For any sub-matching $\M$ of $[n]_0$, the {\it standard form} of $\M$ is a list of blocks $\{(a_1,b_1),(a_2,b_2),\ldots,(a_k,b_k)\}$
such that $a_i<b_i$ for all $i=1,\ldots,k$ and $a_1<a_2<\cdots<a_k$. Let $\SSS(\M)=\{a_1,b_1,a_2,b_2,\ldots,a_k,b_k\}$.
In the following discussion, we always write $\M$ in the standard form.

Let $T\in\mathcal{T}_n$. We now define the following {\it tree-matching algorithm}.

\noindent{\bf Tree-matching algorithm:}
\begin{itemize}
\item Step 1. Let $(a_1,b_1)=(0,1)$.
\item Step 2. At time $k\geq 2$, suppose that $\{(a_1,b_1),(a_2,b_2),\ldots,(a_{k-1},b_{k-1
})\}$ are determined. Denote by $U_k$ the set of vertices $v$ such that $v\notin\{a_1,b_1,a_2,b_2,\ldots,a_{k-1},b_{k-1}\}$ and
 $\child_T(v)\neq\emptyset$.
Then we let $a_k=\min U_k$ and $b_k=\min \Child_T(a_k)$. Iterating Step $2$ until $U_{k+1}=\emptyset$ for some $k$, then
we get a sub-matching $\M=\{(a_1,b_1),(a_2,b_2),\ldots,(a_k,b_k)\}$ of $[n]_0$, and this sub-matching is named {\it tree-matching} and it is denoted by $\M_T$.
\end{itemize}
Let $\M_T=\{(a_1,b_1),(a_2,b_2),\ldots,(a_k,b_k)\}$ be a {\it tree-matching}, and we call the block $(a_i,b_i)$ a {\it tree-pair}.
If $\child_T(a_i)+\child_T(b_i)-1$ is even (resp.~odd), then we say that $(a_i,b_i)$ is an even (resp.~odd) tree-pair.
If $\child_T(a_i)+\child_T(b_i)-1=0$, then we say that $(a_i,b_i)$ is a {\it zero tree-pair}, which is also an even tree-pair.
If $\child_T(a_i)+\child_T(b_i)-1>0$ and $$v=\max \left(\Child_T(a_i)\cup \Child_T(b_i)\setminus\{b_i\}\right),$$ then
we say that $(a_i,b_i)$ is a {\it descent} (resp.~{\it ascent) tree-pair} if $a_i$ (resp.~$b_i$) is the predecessor of the vertex $v$.
For any $v\notin \SSS(\M_T)$, it is called a {\it tree-singleton}, and it is clear that $v\in \leaff_T$.

Let $T\in\mathcal{T}_n$ and let $\M_T$ be the tree-matching of $T$. Let $\st(T)$ denote the number of tree-singletons in $T$.
Let $\zero(T)$ (resp. $\des^e(T)$, $\des^o(T)$, $\asc^e(T)$, $\asc^o(T)$, $\evenp(T)$) denote the number of zero tree-pairs
(resp.~even descent tree-pairs, odd descent tree-pairs, even ascent tree-pairs,
odd ascent tree-pairs, even tree-pairs) in $\M_T$.
It is clear that
\begin{align*}
&\evenp(T)=\zero(T)+\des^e(T)+\asc^e(T),\\
&2(\evenp(T)+\des^o(T)+\asc^o(T))+\st(T)=n+1.
\end{align*}

\begin{ex}\label{ex03}
Consider the following increasing tree $T\in \mathcal{T}_9$:
\begin{center}\includegraphics[width=4cm]{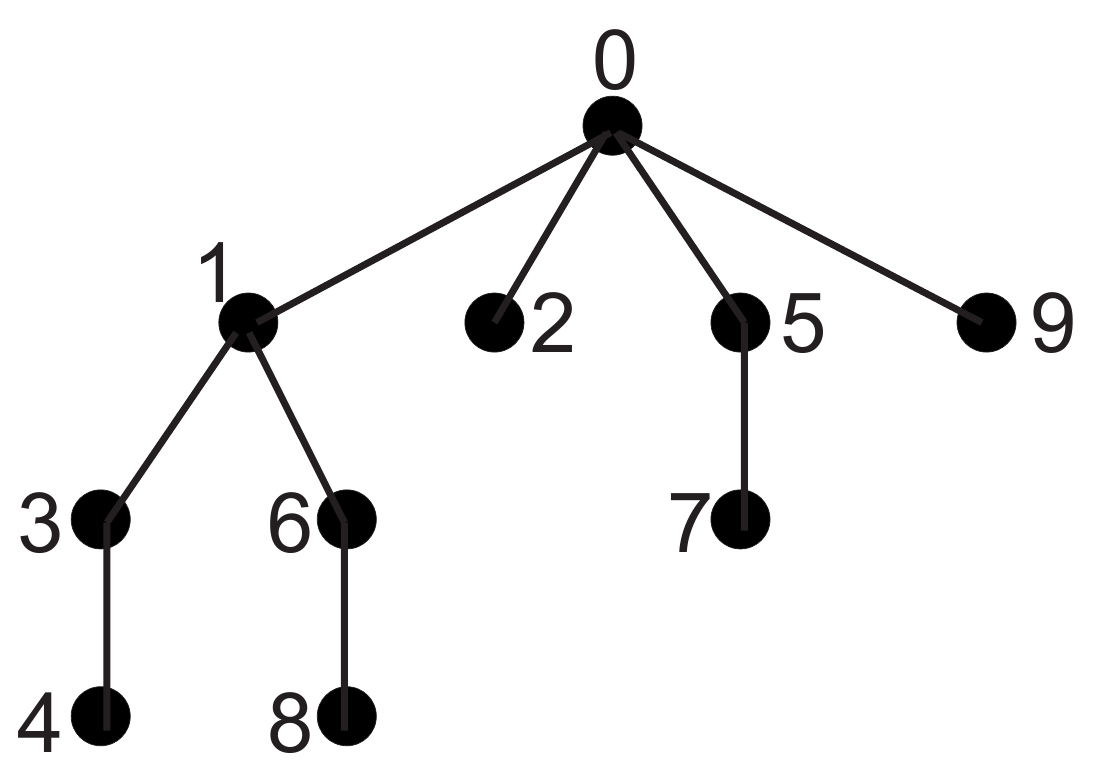}\end{center}
Using the tree-matching algorithm, we obtain $\M_T=\{(0,1),(3,4),(5,7),(6,8)\}$. There are four tree-pairs: $(0,1)$ is an odd descent tree-pair;
$(3,4)$,$(5,7)$ and $(6,8)$ are all zero tree-pairs. The vertices $2$ and $9$ are both tree-singletons.
\end{ex}

For an alphabet $A$, let $\mathbb{Q}[[A]]$ be the rational commutative ring of formal power
series in monomials formed from letters in $A$. A {\it Chen's grammar} (which is also known as context-free grammar) over
$A$ is a function $G: A\rightarrow \mathbb{Q}[[A]]$ that replaces a letter in $A$ by an element of $\mathbb{Q}[[A]]$.
The formal derivative $D_G$ is a linear operator defined with respect to a context-free grammar $G$. 
Following~\cite{Chen17}, a {\it grammatical labeling} is an assignment of the underlying elements of a combinatorial
structure with variables, which is consistent with the substitution rules of a grammar.
The reader is referred to~\cite{Chen17} for more details on this subject.

Note that the differential operator $D_J$ is equivalent to the {\it Schett-Dumont grammar}:
\begin{equation}\label{grammar01}
G=\{x\rightarrow yz,y\rightarrow xz, z\rightarrow xy\}.
\end{equation}
Consider the grammar
\begin{equation}\label{grammar02}
G_1=\{x\rightarrow c,a\rightarrow 2xc, b\rightarrow 2xc,c \rightarrow x(a+b)\}.
\end{equation}
Clearly, the grammar $G_1$ is equivalent to the differential operator $D_G$ given by~\eqref{diff-elliptic02}.
Consider the grammar
\begin{equation}\label{grammar0oe2}
G_2=\{x\rightarrow c,a\rightarrow x(g+h), b\rightarrow x(g+h),c \rightarrow x(a+b),g\rightarrow x(a+b),h \rightarrow x(a+b)\}.
\end{equation}
In particular,
$D_{G_2}^0(x)=x,~
D_{G_2}(x)=c,~
D_{G_2}^2(x)=x(a+b),~
D_{G_2}^3(x)=c(a+b)+2(g+h)x^2$.

\begin{lemma}\label{lemma5}
For the grammar $G_2$, we have
$$D_{G_2}^n(x)=\sum\limits_{T\in\mathcal{T}_n}x^{\st(T)}c^{\zero(T)}a^{\des^o(T)}b^{\asc^o(T)}g^{\des^e(T)}h^{\asc^e(T)}.$$
\end{lemma}
\begin{proof}
We first present a grammatical labeling for $T\in\mathcal{T}_n$.
A tree-singleton $v$ of $T$ is labeled by $x$.
If $(a_i,b_i)$ is a zero tree-pair (resp. odd descent tree-pair, odd ascent tree-pair, even descent tree-pair, even ascent tree-pair)
of $\M_T$, then the edge $(a_i,b_i)$ is labeled by $c$ (resp.~$a,b,g,h$).

Let $T_{(n+1)}$ be a tree generated from $T$ by adding the vertex $n+1$. We distinguish six cases:
\begin{enumerate}
  \item [$(i)$] If we add the vertex $n+1$ as child of a tree-singleton $v$, then $(v,n+1)$ is a zero
tree-pair in $\M_{T_{(n+1)}}$. This corresponds to the substitution rule $x\rightarrow c$ in ${G_2}$.
  \item [$(ii)$] Let $(a_i,b_i)$ be an odd descent tree-pair of $\M_T$. If we add the vertex $n+1$ as child of $a_i$ (resp. $b_i$), then
the number $\child_T(a_i)+\child_T(b_i)-1$ becomes even, and the label of $(a_i,b_i)$ changes from $a$ to $g$ (resp. $h$). Moreover, the vertex $n+1$ is labeled by $x$.
This corresponds to the substitution rule $a\rightarrow x(g+h)$ in ${G_2}$.
  \item [$(iii)$] Let $(a_i,b_i)$ be an odd ascent tree-pair of $\M_T$. If we add the vertex $n+1$ as child of $a_i$ (resp. $b_i$), then
the number $\child_T(a_i)+\child_T(b_i)-1$ becomes even, and the label of $(a_i,b_i)$ changes from $b$ to $g$ (resp. $h$). Moreover, the vertex $n+1$ is labeled by $x$.
This corresponds to the substitution rule $b\rightarrow x(g+h)$ in ${G_2}$.
 \item [$(iv)$] Let $(a_i,b_i)$ be an even descent tree-pair of $\M_T$. If we add the vertex $n+1$ as child of $a_i$ (resp. $b_i$), then
the number $\child_T(a_i)+\child_T(b_i)-1$ becomes odd, and the label of $(a_i,b_i)$ changes from $g$ to $a$ (resp. $b$). Moreover, the vertex $n+1$ is labeled by $x$.
This corresponds to the substitution rule $g\rightarrow x(a+b)$ in ${G_2}$.
 \item [$(v)$] Let $(a_i,b_i)$ be an even ascent tree-pair of $\M_T$. If we add the vertex $n+1$ as child of $a_i$ (resp. $b_i$), then
the number $\child_T(a_i)+\child_T(b_i)-1$ becomes odd, and the label of $(a_i,b_i)$ changes from $h$ to $a$ (resp. $b$). Moreover, the vertex $n+1$ is labeled by $x$.
This corresponds to the substitution rule $h\rightarrow x(a+b)$ in ${G_2}$.
 \item [$(vi)$] Let $(a_i,b_i)$ be a zero tree-pair of $\M_T$. If we add the vertex $n+1$ as child of $a_i$ (resp. $b_i$), then
$\child_T(a_i)+\child_T(b_i)-1=1$, and the label of $(a_i,b_i)$ changes from $c$ to $a$ (resp. $b$). Moreover, the vertex $n+1$ is labeled by $x$.
This corresponds to the substitution rule $c\rightarrow x(a+b)$ in ${G_2}$.
\end{enumerate}

It is routine to check that the action of $D_{G_2}$ on increasing
trees in $\mathcal{T}_n$ generates all the increasing trees in $\mathcal{T}_{n+1}$.
By induction, we see that the above grammatical labeling leads to the desired result.
\end{proof}

Setting $g=h=c$ in Lemma~\ref{lemma5}, then the grammar~\eqref{grammar0oe2} reduces to~\eqref{grammar02}, which leads to
the following result.
\begin{lemma}\label{lemma02}
For the grammar $G_1$ defined by~\eqref{grammar02}, we have
$$D_{G_1}^n(x)=\sum\limits_{T\in\mathcal{T}_n}x^{\st(T)}c^{\evenp(T)}a^{\des^o(T)}b^{\asc^o(T)}.$$
\end{lemma}

Combining~\eqref{recurrence01},~\eqref{abcxyz} and Lemma~\ref{lemma02}, we immediately obtain the following result.
\begin{theorem}
For $n\geq 1$, we have
\begin{align*}
s_{2n,i,j}&=\#\{T\in\mathcal{T}_{2n}\mid \st(T)=2i+1,~\evenp(T)+2\des^o(T)=2j\},\\
s_{2n+1,i,j}&=\#\{T\in\mathcal{T}_{2n+1}\mid \st(T)=2i,~\evenp(T)+2\des^o(T)=2j+1\}.
\end{align*}
\end{theorem}

In the rest of this section, we explore the combinatorial interpretation of $\gamma$-coefficients $\gamma_{n,i,j}$ given by~\eqref{gammapeak}.
For $\M=\{(a_1,b_1),(a_2,b_2),\ldots,(a_k,b_k)\}$,
let $$\mathbb{T}_{\M}=\{T\in \mathcal{T}_n\mid \text{$\M$ is the tree-matching of $T$}\}.$$
For $1\leq i\leq k$, we define a function $\varphi_{\M;(a_i,b_i)}$ on $\mathbb{T}_{\M}$ as follows:
\begin{itemize}\item If $\Child_T(a_i)\cup \Child_T(b_i)\setminus\{b_i\}=\emptyset$, then let $\varphi_{\M;(a_i,b_i)}(T)=T$.
\item Otherwise, let
$v=\max \left(\Child_T(a_i)\cup \Child_T(b_i)\setminus\{b_i\}\right)$.
If $a_i=p_T(v)$, then let $\varphi_{\M;(a_i,b_i)}(T)$ be an increasing tree in $\mathcal{T}_n$ obtained from $T$ by deleting the edge $(a_i,
    v)$ and adding the edge $(b_i,v
    )$; If $b_i=p_T(v)$, then let $\varphi_{\M;(a_i,b_i)}(T)$ be an increasing tree in $\mathcal{T}_n$ obtained from $T$ by deleting the edge $(b_i,
    v)$ and adding the edge $(a_i,v)$.
\end{itemize}
Since the tree-matching of $\varphi_{\M;(a_i,b_i)}(T)$ is still $\M$, we have $\varphi_{\M;(a_i,b_i)}(T)\in\mathbb{T}_{\M}$. Moreover, it is clear that the functions $\varphi_{\M;(a_i,b_i)}$ are all involutions and that they commute. Hence,
for any subset $S \subseteq[k]$, we define the function $\varphi_{\M,S}$ by
$$\varphi_{\M,S}(T)=\prod\limits_{i\in S}\varphi_{\M;(a_i,b_i)}(T).$$
\begin{ex}
Let $\M=\{(0,1),(3,4),(5,7),(6,8)\}$ be a sub-matching of $[9]_0$. Then $\M$ is
the tree-matching of the increasing tree $T$ given in Example~\ref{ex03}. Hence
$T\in\mathbb{T}_{\M}$. Since $$\Child_T(3)\cup \Child_T(4)\setminus\{4\}=\emptyset,$$
we have $\varphi_{\M;(3,4)}(T)=T$. Note that $\max\Child_T(0)\cup \Child_T(1)\setminus\{1\}=\max\{2,3,5,6,9\}=9$. Thus $\varphi_{\M;(0,1)}(T)$ is given as follows:
\begin{center}\includegraphics[width=4cm]{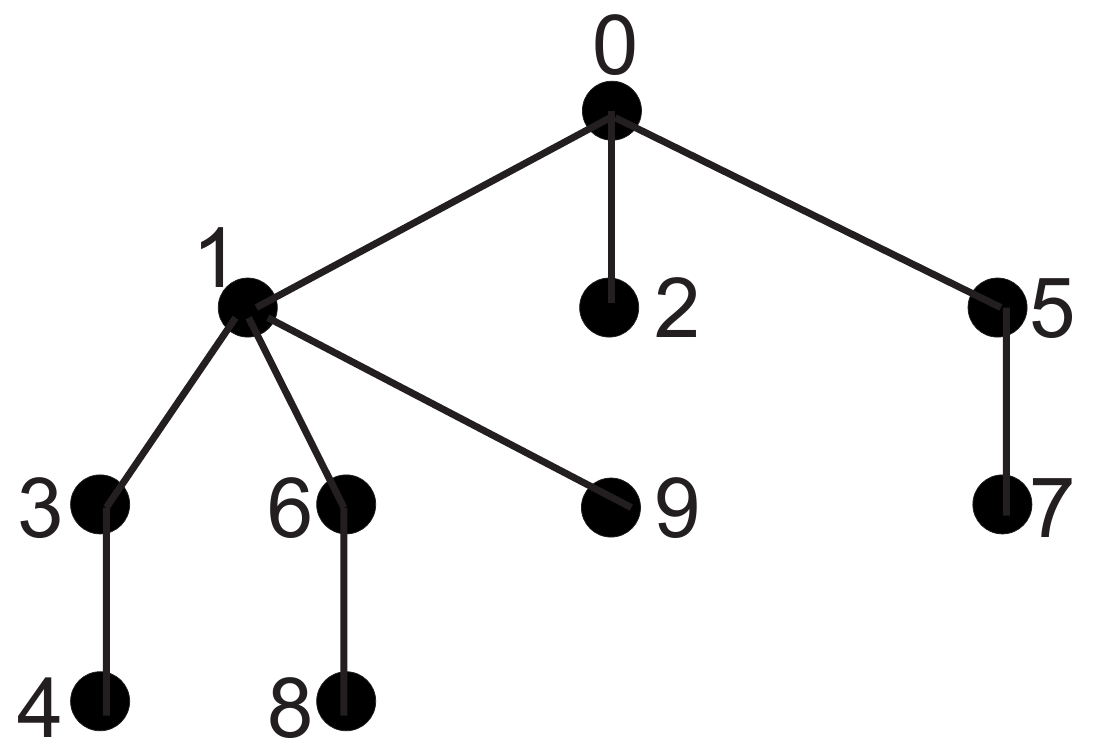}
\end{center}
\end{ex}

By the definition of $\varphi_{\M,S}(T)$, it is easy to verify the following lemma.
\begin{lemma}\label{lemma9}
 For any $T\in \mathcal{T}_n$, suppose that $\M=\{(a_1,b_1),(a_2,b_2),\ldots,(a_{k},b_{k})\}$ is the tree-matching of $T$ with $\asc^o(T)=0$.
Let $S_0=\{l\in [k]: (a_l,b_l)\text{ is an even tree-pair}\}$ and $$S_1=\{l\in [k]\mid (a_l,b_l)\text{ is an odd tree-pair}\}.$$  Then
$\#S_0+\#S_1=k$. For any $A\subseteq S_0$ and $B\subseteq S_1$, let $S=A\cup B$.
Then
\begin{align*}
\st(\varphi_{\M,S}(T))&=\st(T),~
\evenp(\varphi_{\M,S}(T))=\#S_0,\\
\des^o(\varphi_{\M,S}(T))&=\des^o(T)-\#B,~
\asc^o(\varphi_{\M,S}(T))=\#B.
\end{align*}
\end{lemma}

We can now present the following result.
\begin{theorem}\label{mainthm03}
For any $n\geq 0$, we have
\begin{eqnarray*}
\sum\limits_{T\in\mathcal{T}_n}x^{\st(T)}c^{\evenp(T)}a^{\des^o(T)}b^{\asc^o(T)}=\sum_{i,j\geq 0}\theta_{n,i,j}x^{n+1-2(i+j)}c^{i}(a+b)^{j},
\end{eqnarray*}
where $\theta_{n,i,j}=\#\{T\in\mathcal{T}_n\mid \evenp(T)=i,\des^o(T)=j,\asc^o(T)=0\}$.
\end{theorem}
\begin{proof}
Define $\NAP_{n,i,j}=\{T\in\mathcal{T}_n\mid \evenp(T)=i,\des^o(T)=j,\asc^o(T)=0\}$.
For any $T\in \NAP_{n,i,j}$, suppose that $\M=\{(a_1,b_1),(a_2,b_2),\ldots,(a_{i+j},b_{i+j})\}$ is the tree-matching of $T$.
Furthermore, let $S_1=\{l\in [i+j]\mid (a_l,b_l)\text{ is an odd tree-pair}\}$. Then $\#S_1=j$.

Let $[T]=\{\varphi_{\M,S}(T)\mid S\subseteq S_1\}$. For any $T'\in[T]$, suppose that $T'=\varphi_{\M,S}(T)$ for some $S\subseteq S_1$.
By Lemma~\ref{lemma9}, we get
$$\st(T')=\st(T),~\evenp(T')=\evenp(T),$$
$$\des^o(T')=\des^o(T)-\#S,~\asc^o(T')=\#S.$$
It is clear that $\{[T]\mid T\in \NAP_{n,i,j}\}$ form a partition of $\mathcal{T}_n$, since
these are the orbits of the group actions induced by the functions $\varphi_{\M,S}(T)$ and
each orbit contains a tree $T\in \NAP_{n,i,j}$ as a representative.
Hence,
\begin{eqnarray*}
&&\sum\limits_{T\in\mathcal{T}_n}x^{\st(T)}c^{\evenp(T)}a^{\des^o(T)}b^{\asc^o(T)}\\
&=&\sum\limits_{i,j\geq 0}\sum\limits_{T\in \NAP_{n,i,j}}\sum\limits_{T'\in[T]}x^{\st(T')}c^{\evenp(T')}a^{\des^o(T')}b^{\asc^o(T')}\\
&=&\sum\limits_{i,j\geq 0}\sum\limits_{T\in \NAP_{n,i,j}}\sum\limits_{S\subseteq S_1}x^{\st(\varphi_{\M,S}(T))}c^{\evenp(\varphi_{\M,S}(T))}a^{\des^o(\varphi_{\M,S}(T))}
b^{\asc^o(\varphi_{\M,S}(T))}\\
&=&\sum\limits_{i,j\geq 0}\sum\limits_{T\in \NAP_{n,i,j}}\sum\limits_{S\subseteq S_1}x^{\st(T)}c^{\evenp(T)}a^{\des^o(T)-\#S}b^{\#S}\\
&=&\sum\limits_{i,j\geq 0}\sum\limits_{T\in \NAP_{n,i,j}}x^{n+1-2(i+j)}c^{i}\sum\limits_{S\subseteq S_1}a^{j-\#S}b^{\#S}\\
&=&\sum\limits_{i,j\geq 0}\sum\limits_{T\in \NAP_{n,i,j}}x^{n+1-2(i+j)}c^{i}(a+b)^j\\
&=&\sum_{i,j\geq 0}\theta_{n,i,j}x^{n+1-2(i+j)}c^{i}(a+b)^{j}.
\end{eqnarray*}
\end{proof}

Let $\gamma_{n,i,j}$ be defined by~\eqref{recurrence002}. Then
combining Lemma~\ref{lemma02} and Theorem~\ref{mainthm03}, we obtain
\begin{align*}
\gamma_{2n,i,j}&=\theta_{2n,2j,n-i-2j},~
\gamma_{2n+1,i,j}=\theta_{2n+1,2j+1,n-i-2j}.
\end{align*}
Therefore, we get the following result.
\begin{cor}\label{Cor15}
For the $\gamma$-coefficients $\gamma_{n,i,j}$, we have
\begin{align*}
\gamma_{2n,i,j}&=\#\{T\in\mathcal{T}_{2n}\mid \evenp(T)=2j,\des^o(T)=n-i-2j,\asc^o(T)=0\},\\
\gamma_{2n+1,i,j}&=\#\{T\in\mathcal{T}_{2n+1}\mid \evenp(T)=2j+1,\des^o(T)=n-i-2j,\asc^o(T)=0\}.
\end{align*}
\end{cor}
\section{Concluding remarks}
In this paper we obtain a fundamental property of the Jacobian elliptic functions. By using the theory of $\gamma$-positivity,
we get that $J_n(x)$ are unimodal for all $n\geq 0$.
A combinatorial proof of the unimodality of $J_n(x)$ would be interesting.

\end{document}